\documentclass[a4paper,11pt]{amsart}
\usepackage{a4wide}
\usepackage{enumerate}
\usepackage[utf8]{inputenc}
\usepackage{url}
\usepackage{booktabs}          %needed for a better control of tables and arrays
\usepackage{float}             %needed for the option H in tables (here, absolutely here!)
\usepackage{comment}
\usepackage{cite}
\usepackage{hyperref}
\IfFileExists{\jobname.TEX}{}{\usepackage{srcltx}}
\usepackage{array}

\numberwithin{equation}{section}

\newcommand{\field}[1]{\mathbb{{#1}}}

\newcommand{\R}{\field{R}}

\newcommand{\N}{\field{{N}}}

\newcommand{\Q}{\field{{Q}}}
\newcommand{\K}{\field{{K}}}

\newcommand{\disc}{\Delta}
\newcommand{\Chi}{{\raisebox{\depth}{$\chi$}}}
\newcommand{\mltc}{\multicolumn}

\renewcommand{\d}{\,\mathrm{d}}
\renewcommand{\Im}{\mathrm{Im}}
\renewcommand{\Re}{\mathrm{Re}}
\renewcommand{\P}{\mathfrak{p}}

\newcommand{\Norm}{\textrm{\upshape N}}

\DeclareMathOperator{\atan}{atan}

\newtheorem{theorem}{Theorem}[section]
\newtheorem{lemma}[theorem]{Lemma}
\newtheorem*{lemma*}{Lemma}

\newtheorem{corollary}[theorem]{Corollary}

\theoremstyle{remark}

\newtheorem*{remark*}{Remark}

\allowdisplaybreaks[4]

\begin{document}
\title[Explicit prime ideal theorem]{Explicit versions of the prime ideal theorem for Dedekind zeta functions under GRH}

\author[L.~Greni\'{e}]{Lo\"{i}c Greni\'{e}}
\address[L.~Greni\'{e}]{Dipartimento di Ingegneria\\
         Universit\`{a} di Bergamo\\
         viale Marconi 5\\
         I-24044 Dalmine
         Italy}
%\thanks{The research of the third author was supported by Grant \#????.}
\email{loic.grenie@unibg.it}

\author[G.~Molteni]{Giuseppe Molteni}
\address[G.~Molteni]{Dipartimento di Matematica\\
         Universit\`{a} di Milano\\
         via Saldini 50\\
         I-20133 Milano\\
         Italy}
%\thanks{The research of the first author was supported by Grant \#????.}
\email{giuseppe.molteni1@unimi.it}

%11R42 = Zeta functions and $L$-functions of number fields
%11Y70 = values of arithmetic functions; tables;

%\date{\today. File name: {\tt \jobname.tex}}

\begin{abstract}
Let $\psi_\K$ be the Chebyshev function of a number field $\K$. Under the Generalized Riemann Hypothesis
we prove an explicit upper bound for $|\psi_\K(x)-x|$ in terms of the degree and the discriminant of
$\K$. The new bound improves significantly on previous known results.
\end{abstract}

\maketitle

\begin{center}
To appear in Math. Comp. 2015.
\end{center}

\section{Introduction}\label{sec:A1}
For a number field $\K$ we denote
\begin{itemize}
\item[] $n_\K$ its dimension,
\item[] $\disc_\K$ the absolute value of its discriminant,
\item[] $r_1$ the number of its real places,
\item[] $r_2$ the number of its imaginary places,
\item[] $d_\K := r_1+r_2-1$.
\end{itemize}
Moreover, throughout this paper $\P$ denotes a nonzero prime ideal of the integer ring $\mathcal{O}_\K$
and $\Norm\P$ its absolute norm. The von Mangoldt function $\Lambda_\K$ is defined on the set of ideals
of $\mathcal{O}_\K$ as $\Lambda_\K(\mathfrak{I}) := \log\Norm\P$ if $\mathfrak{I}=\P^m$ for some $\P$ and
$m\in\N_{>0}$, and is zero otherwise. Moreover, the function $\pi_\K$ and the Chebyshev function
$\psi_\K$ are defined as
\[
\pi_\K(x)  := \sharp \{\P\colon \Norm\P \leq x\}
\]
and
\[
\psi_\K(x)
:= \sum_{\substack{\mathfrak{I}\subset \mathcal{O}_\K\\ \Norm\mathfrak{I} \leq x}} \Lambda_\K(\mathfrak{I})
 = \sum_{\substack{\P,\,m\\ \Norm\P^m \leq x}} \log \Norm\P.
\]
The original prime number theorem states that
\[
\pi_\Q(x) \sim \frac{x}{\log x}
\qquad \text{as $x\to \infty$}
\]
and was independently proved in~1896 by Ha\-da\-mard and de la Vall\'{e}e--Poussin, both following the ideas
of Riemann. By the work of Chebyshev this claim is equivalent to
\[
\psi_\Q(x)\sim x
\qquad \text{as $x\to \infty$.}
\]
The remainder in these asymptotic behaviors is strictly controlled by the distribution of the nontrivial
zeros of the Riemann zeta function. This was first suggested by Riemann himself, and then confirmed by de
la Vall\'{e}e--Poussin in 1899, when he deduced the now standard estimate for the remainder from the
classical zero free region for the Riemann zeta function. Actually, the Riemann Hypothesis
\[
\zeta(s) \neq 0
\qquad
\forall\,\Re(s)>1/2
\]
is equivalent to the statements
\[
\Big|\pi_\Q(x) - \int_2^{x} \frac{\d u}{\log u}\Big| \ll \sqrt{x}\log x
\]
and
\[
|\psi_\Q(x) - x| \ll \sqrt{x}\log^2 x,
\]
as proved by von Koch in the first years of the twentieth century. A quantitative version of the von Koch
result was proved by Schoenfeld~\cite{Schoenfeld1} in 1976: as a consequence of his previous work in
collaboration with Rosser~\cite{RosserSchoenfeld2} he showed that
\begin{equation}\label{eq:E1}
|\psi_\Q(x) - x| \leq \frac{1}{8\pi}\sqrt{x}\log^2 x
\qquad
\forall\,x\geq 73.2.
\end{equation}
The arguments of Hadamard and de la Vall\'{e}e--Poussin were quickly adapted by Landau to prove analogous
results for a generic number field $\K$, and in 1977 Lagarias and Odlyzko~\cite{LagariasOdlyzko} modified
the argument to explore the dependence of the remainder with respect to the parameters $\disc_\K$ and
$n_\K$. As a part of a more general result on Chebotarev's theorem, they proved that if $\zeta_\K$
satisfies the Generalized Riemann Hypothesis
\[
\zeta_\K(s) \neq 0
\qquad
\forall\,\Re(s)>1/2,
\]
then
\[
|\psi_\K(x) - x| \ll \sqrt{x}[\log x\log\disc_\K + n_\K\log^2 x],
\]
where the implicit constant is independent of $\K$.
%(Actually, their paper deals with the more general Chebotarev's Theorem and proves also unconditional
%results, but here we are mainly interested in this particular result.)
%
Oesterl\'{e} repeated their argument, aiming to produce an explicit value of the absolute constants involved,
and he proved that
\begin{equation}\label{eq:E2}
|\psi_\K(x) - x| \leq \sqrt{x}
                              \Big[
                                   \Big(\frac{\log   x}{ \pi}+2\Big)\log\disc_\K
                                  +\Big(\frac{\log^2 x}{2\pi}+2\Big)n_\K
                              \Big]
\qquad
\forall\,x\geq 1
\end{equation}
under GRH. This result was announced in~\cite{Oesterle}, but unfortunately its proof has never appeared.
Very recently Winckler~\cite[Th.~8.1]{Winckler} has also produced an explicit version of Lagarias and
Odlyzko's work, and proved a result similar to~\eqref{eq:E2}, but with $\frac{23}{3}$ and $\frac{863}{31}$
as coefficients of logs in the $\log\disc_\K$ and $n_\K$ parts, respectively.
\smallskip\\
In this paper we combine a new method to estimate convergent sums on zeros (see Lemma~\ref{lem:E3}), a
very recent result of Trudgian~\cite{TrudgianIII} on the number of zeros in the critical strip and up to
$\pm T$, and an idea of Goldston~\cite{Goldston3}, to deduce the following general result.
\begin{theorem}\label{th:E1}
(GRH)
For every $x\geq 3$ and $T\geq 5$ we have:
\begin{equation}\label{eq:E3}
|\psi_\K(x)-x|\leq
      F(x,T)\log\disc_\K
    + G(x,T)n_\K
    + H(x,T)
\end{equation}
with
\begin{align}
F(x,T) &= \frac{\sqrt{x}}{\pi}
          \Big[\log\Big(\frac{T}{2\pi}\Big)
               + 6.01
%G 4.01 + 2
               + \frac{5.84}{T}
               + \frac{5.52}{T^2}
          \Big]
          + 1.02,                                                                           \label{eq:E4}\\[.3cm]
G(x,T) &= \frac{\sqrt{x}}{\pi}
          \Big[\frac{1}{2}\log^2\Big(\frac{T}{2\pi}\Big)
               + \Big(2 + \frac{5.84}{T} + \frac{5.52}{T^2}\Big)\log\Big(\frac{T}{2\pi}\Big)
               - 1.41
               + \frac{29.04}{T}
               + \frac{31.46}{T^2}
          \Big]
          - 2.10,                                                                           \notag\\[.3cm]
H(x,T) &= \frac{x}{T}
          + \frac{\sqrt{x}}{\pi}\Big[25.57 + \frac{25.97}{T} + \frac{28.57}{T^2}\Big]
          + \epsilon_\K(x,T)
          + 8.35
          + 1.22\frac{\delta_{n_\K\leq 2}}{x}                                               \notag,
\end{align}
where $\epsilon_\K(x,T) := \max\big(0,d_\K\log x - 1.44n_\K\frac{\sqrt{x}}{T}\big)$ and $\delta$ is the
Kronecker symbol.
%G epsK(T,x,nK,dK=nK-1)=max(0,dK*log(x)-1.44*nK*sqrt(x)/T);
%G
%G A =4.01;
%G B =-1.41;
%G C =25.57;
%G FT(T) = log(T/2/Pi) + 2 + A + 5.84/T + 5.52/T^2;
%G F(x,T) = sqrt(x)/Pi*FT(T) + 1.02;
%G GT(T) = log(T/2/Pi)^2/2 + (2 + 5.84/T + 5.52/T^2)*log(T/2/Pi) + B + 29.04/T + 31.46/T^2;
%G G(x,T) = sqrt(x)/Pi*GT(T) - 2.10;
%G H(x,T,nK) = x/T + sqrt(x)/Pi*(C + 25.97/T + 28.57/T^2) + epsK(T,x,nK) + 8.35 + (nK<=2)*1.22/x;
\end{theorem}

\noindent Setting $T$ to a constant one gets a bound of Chebyshev kind, with a main term independent of
the parameters of the field; as a consequence the resulting bound is very strong when $x$ is small with
respect to the degree or the discriminant.\\
Setting $T= x/6$ one gets~\eqref{eq:E2}, for $x\geq 105$ for any non-rational field. By taking $T=8$, the
range can be extended for $x\in[20,105]$: it follows immediately for $n_\K=2$ and $\disc_\K\geq 767842$,
$n_\K=3$ and $\disc_\K\geq 5700$ or $n_\K\geq 4$; the remaining cases for quadratic and cubic fields can
be checked by explicit computations.\\
Comparing the main increasing term $\sqrt{x}\log^2\big(\frac{T}{2\pi}\big)$ with the main decreasing term
$\frac{x}{T}$ we are led to use $T(x) = c\frac{\sqrt{x}}{\log x}$ for suitable values of $c$. In fact,
combining different choices for $c$ we get the following result, which improves significantly
on~\eqref{eq:E2}.
\begin{corollary}\label{cor:E1}
(GRH) Suppose $x\geq 100$. Then
\begin{equation}\label{eq:E5}
|\psi_\K(x)-x| \leq \sqrt{x}\Big[\Big(\frac{\log   x}{2\pi} + 2\Big)\log\disc_\K
                               + \Big(\frac{\log^2 x}{8\pi} + 2\Big)n_\K
                            \Big].
\end{equation}
\end{corollary}
\noindent The range $x\geq 100$ can be extended for fields of large degree, in particular one has $x\geq
24$ when $n_\K\geq 8$, $x\geq 29$ for $n_\K= 7$, $x\geq 43$ for $n_\K = 6$ and $x\geq 72$ for $n_\K=5$.
Only small improvements are possible for cubic and quadratic fields with this method, and only at the
cost of a very large quantity of numerical computations.
\smallskip\\
A different choice of $c$ yields even better results for large $x$.
\begin{corollary}\label{cor:E2}
(GRH) For every $x\geq 3$, we have
\begin{align*}
|\psi_\K(x)-x|
\leq
      \sqrt{x}\Big[\Big(\frac{1}{2\pi}\log\Big(\frac{18.8\,x}{\log^2 x}\Big)
                        + 2.3
                   \Big)\log\disc_\K
                  + \Big(\frac{1}{8\pi}\log^2\Big(\frac{18.8\,x}{\log^2 x}\Big)
                          + 1.3
                     \Big)n_\K
                  + 0.3\log x
                  + 14.6
              \Big].
\end{align*}
Moreover, if $x\geq 2000$, then
\begin{align*}
|\psi_\K(x)-x|
\leq
      \sqrt{x}\Big[\Big(\frac{1}{2\pi}\log\Big(\frac{x}{\log^2 x}\Big)
                        + 1.8
                   \Big)\log\disc_\K
                  +\Big(\frac{1}{8\pi}\log^2\Big(\frac{x}{\log^2 x}\Big)
                  + 1.1
                   \Big)n_\K
                  + 1.2\log x
                  + 10.2
              \Big].
\end{align*}
\end{corollary}
\noindent The first bound is stronger than~\eqref{eq:E2} for $x\geq 1700$ if $\K\neq \Q$ (but $x\geq 280$
suffices when $n_\K\geq 3$ and $x\geq 115$ when $n_\K\geq 4$), and stronger than~\eqref{eq:E5} for $x\geq
1.4\cdot 10^{16}$ (but $x\geq 5.6\cdot 10^{10}$ suffices when $n_\K\geq 3$ and $x\geq 2.2\cdot 10^{8}$
when $n_\K\geq 4$).\\
The second bound is always stronger than~\eqref{eq:E2} when $\K\neq \Q$ and stronger than~\eqref{eq:E5}
for $x\geq 1.4\cdot 10^{32}$ (but $x\geq 9.3\cdot 10^{10}$ suffices when $n_\K\geq 3$ and $x\geq 6.3\cdot
10^5$ when $n_\K\geq 4$; the bad behavior for quadratic fields comes from the term $1.2\log x$). It is
also stronger than~\eqref{eq:E1}, but only for extremely large $x$ (actually $x\geq 3\cdot 10^{871}$).
This is a consequence of the fact that our computations have not been optimized for $\Q$: actually this
is possible in several steps and we believe that doing so the method should produce a better bound.
\medskip

\noindent
From Corollary~\ref{cor:E2} one quickly deduces the following explicit bound for the remainder of the
$\pi_\K(x)$ function.
\begin{corollary}\label{cor:E3}
(GRH) For $x \geq \bar{x} \geq 3$ we have
\begin{multline*}
\Big|\pi_\K(x)-\pi_\K(\bar{x})- \int_{\bar{x}}^x \frac{\d u}{\log u}\Big|\\
\leq
\sqrt{x}\Big[ \Big(\frac{1}{2\pi}-\frac{\log\log x}{ \pi\log x} + \frac{5.8}{\log x}\Big)\log \disc_\K
             +\Big(\frac{1}{8\pi}-\frac{\log\log x}{2\pi\log x} + \frac{3  }{\log x}\Big)n_\K\log x
             + 0.3
             + \frac{13.3}{\log x}
        \Big].
\end{multline*}
\end{corollary}
\smallskip

We have made available at the address:\\
{
\url{http://users.mat.unimi.it/users/molteni/research/psi_GRH/psi_GRH_data.gp}\\
}
a file containing the PARI/GP~\cite{PARI2} code we have used to compute the constants in this article.
\medskip\\
{\bf Acknowledgments.} We wish to thank Alberto Perelli for his valuable remarks and comments, and
Michael Rubinstein who provided us the necessary zeros for a lot of Dirichlet $L$-functions we have used
to check the strength of our Lemma~\ref{lem:E3}. A special thank you goes to Timothy S. Trudgian, who
helped us with illuminating discussions on his work. At last, we thank the referee for her/his
suggestions.

\section{Preliminary inequalities}\label{sec:A2}
For $\Re(s) > 1$ we have
\[
-\frac{\zeta_\K'}{\zeta_\K}(s)
= \sum_\P  \sum_{m=1}^\infty \log(\Norm\P)(\Norm\P)^{-ms},
\]
which in terms of standard Dirichlet series reads
\[
-\frac{\zeta_\K'}{\zeta_\K}(s)
= \sum_{n=1}^{\infty} \tilde{\Lambda}_\K(n)n^{-s},
\quad
\text{with}
\quad
\tilde{\Lambda}_\K(n)
:=
\begin{cases}
\displaystyle
\sum_{\P|p,\, f_\P|k} \log \Norm\P & \text{if $n=p^k$}\\
\quad\ 0                           & \text{otherwise},
\end{cases}
\]
where $f_\P$ is the residual degree of $\P$. The definition of $\tilde{\Lambda}_\K$ shows that
$\tilde{\Lambda}_\K(n)\leq n_\K \Lambda(n)$ for every integer $n$.\\
Let
\begin{equation}\label{eq:E6}
\Gamma_\K(s) := \Big[\pi^{-\frac{s+1}{2}}\Gamma\Big(\frac{s+1}{2}\Big)\Big]^{r_2}
                \Big[\pi^{-\frac{s}{2}}  \Gamma\Big(\frac{s}{2}\Big)  \Big]^{r_1+r_2}
\end{equation}
and
\begin{equation}\label{eq:E7}
\xi_\K(s) := s(s-1) \disc_\K^{s/2}\Gamma_\K(s)\zeta_\K(s),
\end{equation}
then the functional equation for $\zeta_\K$ reads
\begin{equation}\label{eq:E8}
\xi_\K(1-s) = \xi_\K(s).
\end{equation}
Moreover, since $\xi_\K(s)$ is an entire function of order $1$ and does not vanish at $s = 0$, we have
\begin{equation}\label{eq:E9}
\xi_\K(s) = e^{A_\K+B_\K s} \prod_\rho \Big(1 - \frac{s}{\rho}\Big) e^{s/\rho}
\end{equation}
for some constants $A_\K$ and $B_\K$, where $\rho$ runs through all the zeros of $\xi_\K(s)$, which are
precisely those zeros $\rho = \beta + i\gamma$ of $\zeta_\K(s)$ for which $0 < \beta < 1$ and are the
so-called ``nontrivial zeros'' of $\zeta_\K(s)$. From now on $\rho$ will denote a nontrivial zero of
$\zeta_\K(s)$. We recall that the zeros are symmetric with respect to the real axis, as a consequence of
the fact that $\zeta_\K(s)$ is real for $s\in\R$.\\
Differentiating~\eqref{eq:E7} and~\eqref{eq:E9} logarithmically we obtain the identity
\begin{equation}\label{eq:E10}
\frac{\zeta'_\K}{\zeta_\K}(s)
  = B_\K
   + \sum_\rho \Big(\frac{1}{s-\rho}+\frac{1}{\rho}\Big)
   - \frac{1}{2}\log \disc_\K
   - \Big[\frac{1}{s}+\frac{1}{s-1}\Big]
   - \frac{\Gamma'_\K}{\Gamma_\K}(s),
\end{equation}
valid identically in the complex variable $s$.\\
Stark~\cite[Lemma~1]{Stark1} proved that the functional equation~\eqref{eq:E8} implies that $B_\K =
-\sum_\rho \Re(\rho^{-1})$ (see also~\cite{Odlyzko5} and~\cite[Ch.~XVII, Th.~3.2]{Lang1}), and that once
this information is available one can use~\eqref{eq:E10} and the definition of the gamma factor
in~\eqref{eq:E6} to prove that the function $f_\K(s) := \sum_\rho\Re\big(\frac{2}{s-\rho}\big)$ can be
exactly computed via the alternative representation
\begin{equation}\label{eq:E12}
f_\K(s)
= 2\Re\frac{\zeta'_\K}{\zeta_\K}(s)
    +\log\frac{\disc_\K}{\pi^{n_\K}}
    +\Re\Big(\frac{2}{s}
             +\frac{2}{s-1}
        \Big)
    +(r_1+r_2)\Re\frac{\Gamma'}{\Gamma}\Big(\frac{s}{2}\Big)
    +r_2\Re\frac{\Gamma'}{\Gamma}\Big(\frac{s+1}{2}\Big).
\end{equation}

Using~\eqref{eq:E7}, \eqref{eq:E8} and~\eqref{eq:E10} one sees that
\begin{equation}\label{eq:E13}
\begin{array}{ll}
\displaystyle\frac{\zeta'_\K}{\zeta_\K}(s) = \frac{r_1+r_2-1}{s} + r_\K + O(s) & \text{as $s\to 0 $}\\[.4cm]
\displaystyle\frac{\zeta'_\K}{\zeta_\K}(s) = \frac{r_2}{s+1} + r'_\K + O(s+1)  & \text{as $s\to -1$},
\end{array}
\end{equation}
where
\[
\begin{array}{ll}
\displaystyle
r_\K  = B_\K
        +1
        -\frac{1}{2} \log\frac{\disc_\K}{\pi^{n_\K}}
        -\frac{r_1+r_2}{2}\frac{\Gamma'}\Gamma(1)
        -\frac{r_2}{2}    \frac{\Gamma'}\Gamma\Big(\frac{1}{2}\Big),\\[.4cm]
\displaystyle
r'_\K  = -\frac{\zeta'_\K}{\zeta_\K}(2)
         -\log\frac{\disc_\K}{\pi^{n_\K}}
        -\frac{n_\K}{2}\frac{\Gamma'}\Gamma\Big(\frac{3}{2}\Big)
        -\frac{n_\K}{2}\frac{\Gamma'}\Gamma(1).
\end{array}
\]
In order to prove our result we need the following explicit bound for $r_\K$
\begin{equation}\label{eq:E14}
|r_\K|\leq  1.02\log \disc_\K - 2.10 n_\K + 8.35
\end{equation}
which is Lemma~3.2 in~\cite{GrenieMolteni2}.\\
At last, we need two elementary lemmas. The first one is an optimized version of a lemma due to
Littlewood~\cite{Littlewood}.
\begin{lemma}\label{lem:E1}
If $x\geq -1$ and $1\leq \Re(\nu)\leq 2$, then
\[
|(1+x)^{\nu}-1-\nu x|
\leq
\Big(\frac{1}{2}
     +\Big(\frac{1}{\Re(\nu)}-\frac{1}{2}\Big)\max(0,-x)
\Big)|\nu(\nu-1)x^2|.
\]
\end{lemma}
\begin{proof}
The statement is obvious for $\nu=1$, $\nu=2$, $x=-1$ and $x=0$, we thus suppose we are in another case.
From the equality $f(x)-f(0)-f'(0)x = \int_{0}^x\int_{0}^u f''(v)\d v \d u$ one gets
\begin{equation}\label{eq:tardo}
\frac{(1+x)^{\nu}-1-\nu x}{\nu(\nu-1)} = \int_{0}^x\int_{0}^u (1+v)^{\nu-2}\d v \d u.
\end{equation}
Let $x\geq 0$, then $|1+v|^{\Re(\nu)-2}\leq 1$ thus
\[
\frac{|(1+x)^{\nu}-1-\nu x|}{|\nu(\nu-1)|} \leq \int_{0}^x\int_{0}^u \d v \d u
= \frac{x^2}{2}.
\]
Let $x\in (-1,0)$. Then
\begin{equation}\label{eq:E15}
\frac{\Re(\nu)(\Re(\nu)-1)}{x^2}
    \Big|\int_{0}^x\Big|\int_{0}^u (1+v)^{\Re(\nu)-2}\d v\Big| \d u\Big|
= \frac{(1+x)^{\Re(\nu)}-1-\Re(\nu)x}{x^2}.
\end{equation}
The right-hand side may be written as $\sum_{k=0}^{\infty}\!\binom{\Re(\nu)}{k+2}x^k$  and its second
derivative as $\sum_{k=0}^{\infty}\!\binom{\Re(\nu)}{k+4}$ $(k+2)(k+1)x^k$. When $x\in(-1,0)$ each term of
the series is positive; this proves that the right-hand side in~\eqref{eq:E15} is convex in $(-1,0)$ so
that its graph is below the line connecting its points with $x=-1$ and $x=0$. Said line has equation
$y=(\frac{1}{2}+(\frac{1}{2}-\frac{1}{\Re(\nu)})x)\Re(\nu)(\Re(\nu)-1)$,
thus~\eqref{eq:E15} gives
\[
\Big|\int_{0}^x\Big|\int_{0}^u (1+v)^{\Re(\nu)-2}\d v\Big| \d u\Big|
\leq \Big(\frac{1}{2}+\Big(\frac{1}{2}-\frac{1}{\Re(\nu)}\Big)x\Big)x^2
\]
for $\Re(\nu)>1$ immediately and for $\Re(\nu)\geq 1$ by continuity. We get the claim comparing it
with~\eqref{eq:tardo}.
\end{proof}

\begin{lemma}\label{lem:E2}
Let
\[
f_1(x) := \sum_{r=1}^{\infty}\frac{x^{1-2r}}{2r(2r-1)},
\qquad
f_2(x) := \sum_{r=2}^{\infty}\frac{x^{2-2r}}{(2r-1)(2r-2)},
\]
\[
R_{r_1,r_2}(x) := - (r_1+r_2-1)(x\log x-x)
                  + r_2(\log x +1)
                  - (r_1+r_2)f_1(x)
                  - r_2      f_2(x).
\]
Let $x\geq 3$, then
\[
-(r_1+r_2-1)\log x \leq R_{r_1,r_2}'(x) \leq 1.22\frac{\delta_{n_\K\leq 2}}{x}.
\]
\end{lemma}
\begin{proof}
We have
\[
f_1(x) = \frac{1}{2}\Big[x\log(1-x^{-2}) + \log\Big(\frac{1+x^{-1}}{1-x^{-1}}\Big)\Big],
\qquad
f_2(x) = 1-\frac{1}{2}\Big[\log(1-x^{-2}) + x\log\Big(\frac{1+x^{-1}}{1-x^{-1}}\Big)\Big],
\]
%G f1(x)=(x*log(1-1/x^2)+log((1+1/x)/(1-1/x)))/2
%G f2(x)=1-(log(1-1/x^2)+x*log((1+1/x)/(1-1/x)))/2
%G \\f2(1/u)
%G \\%68 = 1/6*u^2 + 1/20*u^4 + 1/42*u^6 + 1/72*u^8 + 1/110*u^10 + 1/156*u^12 + 1/210*u^14 + O(u^15)
%G \\f1(1/u)
%G \\%69 = 1/2*u + 1/12*u^3 + 1/30*u^5 + 1/56*u^7 + 1/90*u^9 + 1/132*u^11 + 1/182*u^13 + O(u^15)
%
%G R(r1,r2,x)=-(r1+r2-1)*(x*log(x)-x) +r2*(log(x)+1) -(r1+r2)*f1(x) -r2*f2(x)
%
and
\begin{align*}
R_{r_1,r_2}'(x)&= -(r_1+r_2-1)\log x
                  - \frac{1}{2}(r_1+r_2)\log(1-x^{-2})
                  - \frac{r_2}{2} \log\Big(\frac{1-x^{-1}}{1+x^{-1}}\Big)\\
               &= -(r_1+r_2-1)\log x
                  - \frac{r_1}{2}\log(1-x^{-2})
                  -       r_2    \log(1-x^{-1});
\end{align*}
this equality already proves the lower bound. The upper bound immediately follows for the cases where
$r_1+r_2=1$. Suppose $r_1+r_2\geq 2$, writing $R_{r_1,r_2}'(x)$ as
\[
R_{r_1,r_2}'(x) = \log x
                  - \frac{r_1}{2}\log(x^2-1)
                  -       r_2    \log(x-1).
\]
Then for $x> 1$ one gets
\[
R_{r_1,r_2}'(x) \leq \log x - (r_1+r_2)\log(x-1)
                \leq \log x - 2\log(x-1)
                \leq 0,
\]
where the last inequality is true for $x\geq \tfrac{3+\sqrt{5}}{2}= 2.61\ldots$
\end{proof}

\section{Upper bounds}\label{sec:A3}
For the proof of the theorem we need bounds for three sums on nontrivial zeros, namely for
\[
\sum_{|\gamma|\leq T} 1,
\qquad
\sum_{|\gamma|\geq T} \frac{1}{|\rho|^2}
\quad
\text{and}
\quad
\sum_{|\gamma|\leq T} \frac{1}{|\rho|}.
\]
The first sum is simply the number $N_\K(T)$ of nontrivial zeros in the rectangle $0<\Re(s)< 1$,
$|\Im(s)|\leq T$. It has been explicitly estimated by Trudgian~\cite{TrudgianIII} in a work improving
Kadiri--Ng's paper~\cite{KadiriNg}. We estimate the second sum by partial summation using this result.
For the last one a simple partial summation is not possible since both Kadiri--Ng's and Trudgian's
results are proved only for $T\geq 1$ and improve when the range is further restricted to $T\geq T_0$
with a $T_0\geq 1$. As a consequence we bound the part of the third sum coming from the zeros far enough
of the real axis by partial summation, and the remaining with a different technique. In fact,
in~\cite{GrenieMolteni2} we have shown a new method to bound converging sums on zeros under GRH. The
method works very well but depends on several parameters whose values are fixed via a trial and error
approach. Thus, in order to apply it we need to fix a value for $T_0$, and the final result will only be
valid in the range $T\geq T_0$. After several tests the choice $T_0=5$ seemed to represent a good
compromise between the need of having a large $T_0$ (to take advantage of the better estimate in
Trudgian's result) and a small $T_0$ (to make the final theorem valid in a larger range). Our result is
as follows.
\begin{lemma}\label{lem:E3}
(GRH) One has
\[
\sum_{|\gamma|\leq 5} \frac{1}{|\rho|} \leq 1.02\log\disc_\K - 1.63 n_\K + 7.04.
\]
\end{lemma}
\begin{proof}
We apply the same technique we have already used for Lemma~4.1 in~\cite{GrenieMolteni2}. Thus, let
$f(s,\gamma):=4(2s-1)/((2s-1)^2+4\gamma^2)$, so that $f_\K(s) = \sum_{\gamma} f(s,\gamma)$, and let
$g(\gamma):= 2(1+4\gamma^2)^{-1/2}\Chi_{[-5,5]}(\gamma)$, so that $\sum_{|\gamma|\leq 5}|\rho|^{-1} =
\sum_{\gamma} g(\gamma)$. We look for a finite linear combination of $f(s,\gamma)$ at suitable points
$s_j$ such that
\begin{equation}\label{eq:E16}
g(\gamma) \leq F(\gamma) := \sum_j a_j f(s_j,\gamma)
\qquad \forall\gamma\in\R,
\end{equation}
so that
\begin{equation}\label{eq:E17}
\sum_{|\gamma|\leq 5}\frac{1}{|\rho|} \leq \sum_j a_j f_\K(s_j);
\end{equation}
once~\eqref{eq:E17} is proved, we recover a bound for the sum on zeros recalling the
identity~\eqref{eq:E12}. According to this approach the final coefficient of $\log\disc_\K$ will be the
sum of all $a_j$, thus we are interested into linear combinations for which this sum is as small as
possible. We set $s_j = 1+ j/2$ with $j=1,\ldots,2q+3$ for a suitable integer $q$. Let $\Upsilon \subset
(0,\infty)$ be a set with $q$ numbers. We require:
\begin{enumerate}
\item $F(\gamma)=g(\gamma)$ for all $\gamma\in\Upsilon\cup\{0,5\}$,
\item $F'(\gamma)=g'(\gamma)$ for all $\gamma\in\Upsilon$,
\item $\lim_{\gamma\to\infty}\gamma^2 F(\gamma)=\lim_{\gamma\to\infty}\gamma^2 g(\gamma)=0$.
\end{enumerate}
This produces a set of $2q+3$ linear equations for the $2q+3$ constants $a_j$, and we hope that these
satisfy~\eqref{eq:E16} for every $\gamma$. We choose $q:=22$ and $\Upsilon:=\{0.6, 1, 1.9, 2.9, 3.9, 10, 13,
14, 15, 16, 17,$ $18, 19, 20, 30, 40, 50, 100, 10^3, 10^4, 10^5, 10^6\}$. Finally, with an abuse of
notation we take for $a_j$ the solution of the system, rounded above to $10^{-7}$: this produces the
numbers in Table~\ref{tab:E4}. Then, using Sturm's algorithm, we prove that the values found actually
give an upper bound for $g$, so that~\eqref{eq:E17} holds with such $a_j$'s. These constants verify
\begin{equation}\label{eq:E18}
\begin{aligned}
&\sum_j a_j                                                =     1.011\ldots,\\
&\sum_j a_j\frac{\Gamma'}{\Gamma}\Big(\frac{s_j}{2}\Big)   \leq -1.13,
\end{aligned}
\qquad\qquad
\begin{aligned}
&\sum_j a_j\Big(\frac{2}{s_j}+\frac{2}{s_j-1}\Big)         \leq  7.04,\\
&\sum_j a_j\frac{\Gamma'}{\Gamma}\Big(\frac{s_j+1}{2}\Big) \leq -0.31.
\end{aligned}
\end{equation}
We write $\sum_j a_j\frac{\zeta'_\K}{\zeta_\K}(s_j)$ as
\[
-\sum_n\tilde\Lambda_\K(n)S(n)\quad\text{with}\quad S(n):=\sum_j \frac{a_j}{n^{s_j}}.
\]
We check numerically that $S(n)>0$ for $n\leq 60975$ and that it is negative for $60975<n\leq 128000$.
Then, since the sign of $a_j$ alternates, we can easily prove that each pair $\frac{a_1}{n^{s_1}} +
\frac{a_2}{n^{s_2}}$, \ldots, $\frac{a_{2q+1}}{n^{s_{2q+1}}} + \frac{a_{2q+2}}{n^{s_{2q+2}}}$ and the
last term $\frac{a_{2q+3}}{n^{s_{2q+3}}}$ are negative for every $n\geq 128000$. Thus
\begin{align}
\sum_j a_j\frac{\zeta'_\K}{\zeta_\K}(s_j)
&= -\sum_n\tilde\Lambda_\K(n)S(n)
 \leq -n_\K\sum_{n> 60975}\Lambda(n)S(n)                        \label{eq:E19}\\
&=    -n_\K\Big[\sum_{n=1}^\infty\Lambda(n)S(n)
                -\sum_{n\leq 60975}\Lambda(n)S(n)
           \Big]                                                \notag\\
&=     n_\K\Big[\sum_{j}a_j\frac{\zeta'}{\zeta}(s_j)
                +\sum_{n\leq 60975}\Lambda(n)S(n)
           \Big]
 \leq 0.12 n_\K.                                                \notag
\end{align}
The result now follows from~\eqref{eq:E12}, and~(\ref{eq:E17}--\ref{eq:E19}).
\end{proof}

Now we can bound the sums.

\subsection*{First sum}
Trudgian~\cite{TrudgianIII} has proved that
\begin{equation}\label{eq:E20}
\Big|N_\K(T)-\frac{T}{\pi}\log\Big(\Big(\frac{T}{2\pi e}\Big)^{n_\K}\disc_\K\Big)\Big|
\leq
\frac{1}{\pi}(c_1(\eta)W_\K(T) + c_2(\eta)n_\K + c_3(\eta))
\qquad \forall T\geq T_0\geq 1,
\end{equation}
where $W_\K(T) := \log\disc_\K+n_\K\log(T/2\pi)$, $c_1(\eta)=\pi D_1$, $c_2(\eta)=\pi(D_2+D_1\log 2\pi)$
and $c_3(\eta)=\pi D_3$ and the $D_j$ are Trudgian's constants which depend on $T_0$,
$\eta\in(0,\tfrac{1}{2}]$ and on two other parameters $p$ and $r$. We thus have
\[
N_\K(T)
\leq \frac{T}{\pi}\Big(1+\frac{c_1(\eta)}{T}\Big)W_\K(T)  \\
      - \frac{T}{\pi}\Big(1 - \frac{c_2(\eta)}{T}\Big)n_\K
      + \frac{c_3(\eta)}{\pi}
\qquad
\forall T\geq T_0.
\]
We fix $\eta=\frac{1}{2}$, $p=-\eta=-\frac{1}{2}$ (this choice differs from the one
in~\cite{TrudgianIII}) and $r=\frac{1+\eta-p}{1/2+\eta}$ (as in~\cite{TrudgianIII}), so that actually
$r=2$; recall that $T_0=5$. Following Trudgian's argument we find $D_1 = 0.459\ldots$, $D_2= 1.996\ldots$,
$D_3= 2.754\ldots$, hence
\begin{equation}\label{eq:E21}
N_\K(T)
\leq \frac{T}{\pi}\Big(1+\frac{1.45}{T}\Big)W_\K(T)
      - \frac{T}{\pi}\Big(1 - \frac{8.93}{T}\Big)n_\K
      + \frac{8.66}{\pi}
\qquad
\forall T\geq 5.
\end{equation}

\subsection*{Second sum}
We proceed by partial summation. Let formula~\eqref{eq:E20} for $N_\K(T)$ be written as $A(T) + R(T)$,
respectively the asymptotic and the remainder term. Then
\begin{align*}
\sum_{|\gamma|\geq T}\frac{1}{|\rho|^2}
&= \sum_{|\gamma|\geq T}\frac{1}{1/4+\gamma^2}
\leq \int_T^{+\infty}\frac{\d A(\gamma)}{1/4+\gamma^2}
   + \frac{R(T)}{1/4+T^2}
   + \int_T^{+\infty}\frac{2\gamma R(\gamma)\d \gamma}{(1/4+\gamma^2)^2}    \\
&= \int_T^{+\infty}\frac{\d A(\gamma)}{1/4+\gamma^2}
   + \frac{2R(T)}{1/4+T^2}
   + \int_T^{+\infty}\frac{R'(\gamma)\d \gamma}{1/4+\gamma^2}               \\
&= \int_T^{+\infty}\frac{\d A(\gamma)}{1/4+\gamma^2}
   + \frac{2R(T)}{1/4+T^2}
   + \frac{c_1(\eta)}{\pi}n_\K\int_T^{+\infty}\frac{\gamma^{-1}\d \gamma}{1/4+\gamma^2} \\
&\leq \int_T^{+\infty}\frac{\d A(\gamma)}{1/4+\gamma^2}
   + \frac{2R(T)}{T^2}
   + \frac{c_1(\eta)}{2\pi T^2}n_\K.
\end{align*}
Using
\[
\int_T^{+\infty}\frac{\d \gamma}{1/4+\gamma^2} = 2\atan\Big(\frac{1}{2T}\Big)
\quad
\text{which is}\leq \frac{1}{T},
\text{ and }\geq \frac{1}{T} - \frac{1/12}{T^3}
\]
\[
\int_T^{+\infty}\frac{\log\gamma}{1/4+\gamma^2}\d \gamma
\leq \int_T^{+\infty}\gamma^{-2}\log\gamma\d \gamma
= \frac{\log(eT)}{T},
\]
one has
\[
\int_T^{+\infty}\frac{\d A(\gamma)}{1/4+\gamma^2}
\leq \frac{W_\K(T)}{\pi T}
    + \Big(1 + \frac{\log 2\pi}{12T^2}\Big)\frac{n_\K}{\pi T}.
\]
Thus
{\small
\begin{align*}
\sum_{|\gamma|\geq T}\frac{\pi}{|\rho|^2}
\leq& \frac{W_\K(T)}{T}
    + \Big(1 + \frac{\log 2\pi}{12T^2}\Big)\frac{n_\K}{T}
    +\frac{2}{T^2}\Big[c_1(\eta)\log\disc_\K + \Big(c_1(\eta)\log\Big(\frac{T}{2\pi}\Big) +c_2(\eta)+ \frac{c_1(\eta)}{4}\Big)n_\K + c_3(\eta)\Big]\\
=& \Big(1 + \frac{2c_1(\eta)}{T}\Big)\frac{W_\K(T)}{T}
    + \Big(1 + \frac{\log 2\pi}{12T^2}\Big)\frac{n_\K}{T}
    +\Big(2c_2(\eta)+ \frac{c_1(\eta)}{2}\Big)\frac{n_\K}{T^2}
    +\frac{2c_3(\eta)}{T^2}
\end{align*}
}
hence
\begin{equation}\label{eq:E22}
\sum_{|\gamma|\geq T}\frac{\pi}{|\rho|^2}
\leq \Big(1 + \frac{2.89}{T}\Big)\frac{W_\K(T)}{T}
    +\Big(1 + \frac{18.61}{T}\Big)\frac{n_\K}{T}
    +\frac{17.31}{T^2}
\qquad
\forall T\geq 5.
\end{equation}

\subsection*{Third sum}
We proceed again by partial summation, plus the contribution of Lemma~\ref{lem:E3} to bound the part of
the sum coming from low-lying zeros. We have
\begin{align*}
\sum_{|\gamma|\leq T}\frac{1}{|\rho|}
&= \sum_{|\gamma|\leq T}\frac{1}{(1/4+\gamma^2)^{1/2}}
\leq \sum_{|\gamma|\leq 5} \frac{1}{|\rho|}
   + \sum_{5\leq |\gamma|\leq T}\frac{1}{(1/4+\gamma^2)^{1/2}}                            \\
&\leq \sum_{|\gamma|\leq 5} \frac{1}{|\rho|}
   + \int_5^{T}\frac{\d A(\gamma)}{(1/4+\gamma^2)^{1/2}}
   + \frac{2R(5)}{\sqrt{101}}
   + \frac{R(T)}{(1/4+T^2)^{1/2}}
   + \int_5^{T}\frac{\gamma R(\gamma)\d \gamma}{(1/4+\gamma^2)^{3/2}}                     \\
&= \sum_{|\gamma|\leq 5} \frac{1}{|\rho|}
   + \int_5^{T}\frac{\d A(\gamma)}{(1/4+\gamma^2)^{1/2}}
   + \frac{4R(5)}{\sqrt{101}}
   + \int_5^{T}\frac{R'(\gamma)\d \gamma}{(1/4+\gamma^2)^{1/2}}                           \\
&= \sum_{|\gamma|\leq 5} \frac{1}{|\rho|}
   + \frac{4R(5)}{\sqrt{101}}
   + \int_5^{T}\frac{\d A(\gamma)}{(1/4+\gamma^2)^{1/2}}
   + \frac{c_1(\eta)}{\pi}n_\K\int_5^{T}\frac{\gamma^{-1}\d \gamma}{(1/4+\gamma^2)^{1/2}} \\
&\leq \sum_{|\gamma|\leq 5} \frac{1}{|\rho|}
   + \frac{4R(5)}{\sqrt{101}}
   + 0.2 \frac{c_1(\eta)}{\pi}n_\K
   + \int_5^{T}\frac{\d A(\gamma)}{(1/4+\gamma^2)^{1/2}}.
\end{align*}
Using
\[
\int_5^{T}\frac{\d \gamma}{(1/4+\gamma^2)^{1/2}}
= \log\Big(\frac{2T+\sqrt{4T^2+1}}{10+\sqrt{101}}\Big)
\]
which is $\leq \log T -\log 5$ and $\geq \log T - 1.62$ for $T\geq 5$, and
%G log((2+sqrt(4))/(10+sqrt(101)))
\[
\int_5^{T}\frac{\log\gamma\d \gamma}{(1/4+\gamma^2)^{1/2}}
\leq \int_5^{T}\frac{\log\gamma}{\gamma} \d \gamma
=    \frac{\log^2 T}{2} - \frac{\log^2 5}{2},
\]
one has
\[
\int_5^{T}\frac{\d A(\gamma)}{(1/4+\gamma^2)^{1/2}}
\leq \Big(\log\Big(\frac{T}{2\pi}\Big) + 0.23\Big)\frac{\log\disc_\K}{\pi}
    + \Big(\log^2\Big(\frac{T}{2\pi}\Big) - 0.01\Big)\frac{n_\K}{2\pi}.
\]
%G log(2*pi)-log(5)
%G \\%384 = 0.2284391539752451089599001396
%G -log(2*Pi)^2+2*log(2*Pi)*1.62 -log(5)^2
%G \\%288 = -0.013360810047377287224830168
%
Thus recalling Lemma~\ref{lem:E3} we get
\begin{align*}
\sum_{|\gamma|\leq T}\frac{\pi}{|\rho|}
\leq& \Big(\log\Big(\frac{T}{2\pi}\Big) + 0.23\Big)\log\disc_\K
    +\frac{n_\K}{2}\Big(\log^2\Big(\frac{T}{2\pi}\Big) - 0.01\Big)
    + \sum_{|\gamma|\leq 5} \frac{\pi}{|\rho|}
    + \frac{4\pi R(5)}{\sqrt{101}}
    + 0.2 c_1(\eta)n_\K                                                              \\
\leq& \Big(\log\Big(\frac{T}{2\pi}\Big) + 0.23\Big)\log\disc_\K
    +\frac{n_\K}{2}\Big(\log^2\Big(\frac{T}{2\pi}\Big) - 0.01\Big)
    + \pi(1.02\log\disc_\K - 1.63 n_\K + 7.04)                                       \\
&   + \frac{4}{\sqrt{101}}\Big(c_1(\eta)(\log\disc_\K + n_\K\log\big(\frac{5}{2\pi}\big)) + c_2(\eta)n_\K + c_3(\eta)\Big)
    + 0.2 c_1(\eta)n_\K                                                              \\
=& \Big(\log\Big(\frac{T}{2\pi}\Big) + 0.23 + 1.02\pi + \frac{4}{\sqrt{101}}c_1(\eta)\Big)\log\disc_\K
    + 7.04\pi + \frac{4}{\sqrt{101}}c_3(\eta)                                        \\
&   + \Big(\frac{1}{2}\log^2\Big(\frac{T}{2\pi}\Big)
           - \frac{1}{2} 0.01
           - 1.63\pi
           + \frac{4\log(5/2\pi)}{\sqrt{101}}c_1(\eta)
           + \frac{4}{\sqrt{101}}c_2(\eta)
           + 0.2 c_1(\eta)
      \Big)n_\K.
\end{align*}
Hence
\begin{equation}\label{eq:E23}
\sum_{\substack{\rho\\|\gamma|\leq T}}\frac{\pi}{|\rho|}
\leq \Big(\log\Big(\frac{T}{2\pi}\Big) + 4.01\Big)\log \disc_\K
     + \Big(\frac{1}{2}\log^2\Big(\frac{T}{2\pi}\Big) - 1.41\Big)n_\K
     + 25.57
\qquad
\forall T\geq 5.
\end{equation}

\section{Proofs}\label{sec:A4}
\begin{proof}[Proof of Theorem~\ref{th:E1}]
Let
\[
\psi^{(1)}_\K(x) := \int_{0}^{x}\psi_\K(t)\d t.
\]
As observed by Goldston~\cite{Goldston3}, since $\Lambda_\K(\mathfrak{I})\geq 0$ one has the double
inequality
\begin{equation}\label{eq:E24}
\begin{array}{l}
\displaystyle
\psi_\K(x) \leq \frac{\psi^{(1)}_\K(x+h)-\psi^{(1)}_\K(x)}{h} \qquad \text{if $h>0$},   \\
\displaystyle
\psi_\K(x) \geq \frac{\psi^{(1)}_\K(x+h)-\psi^{(1)}_\K(x)}{h} \qquad \text{if $-x<h<0$}.
\end{array}
\end{equation}
As in~\cite[Ch.~IV Sec.~4, p. 73]{Ingham2} and~\cite[Sec.~5]{LagariasOdlyzko}, considering the
integral representation
\[
\psi^{(1)}_\K(x) = -\frac{1}{2\pi i}\int_{2-i\infty}^{2+i\infty} \frac{\zeta'_\K}{\zeta_\K}(s)\frac{x^{s+1}}{s(s+1)}\,\d s
\]
one gets for every $x>1$ the identity
\[
\psi^{(1)}_\K(x) = \frac{x^2}{2}
              - \sum_{\rho}\frac{x^{\rho+1}}{\rho(\rho+1)}
              - x r_\K
              + r'_\K
              + R_{r_1,r_2}(x)
\]
where $R_{r_1,r_2}(x)$ is defined in Lemma~\ref{lem:E2} and $r_\K$ and $r'_\K$ are defined
in~\eqref{eq:E13}. Thus
\[
\frac{\psi^{(1)}_\K(x+h)-\psi^{(1)}_\K(x)}{h}
= x
  + \frac{h}{2}
  - \sum_{\rho}\frac{(x+h)^{\rho+1}-x^{\rho+1}}{h\rho(\rho+1)}
  - r_\K
  + R_{r_1,r_2}'(\eta)
\]
for a suitable $\eta$ in the interval between $x$ and $x+h$. Hence, for every $x\geq 3$ and $h\neq 0$
such that $x+h>1$, Lemma~\ref{lem:E2} gives
\begin{equation}\label{eq:E25}
-d_\K\log x
\leq
\frac{\psi^{(1)}_\K(x+h)-\psi^{(1)}_\K(x)}{h}
- \Big(
        x+\frac{h}{2}
        - \sum_{\rho}\frac{(x+h)^{\rho+1}-x^{\rho+1}}{h\rho(\rho+1)}
        - r_\K
  \Big)
\leq
1.22\frac{\delta_{n_\K\leq 2}}{x}.
\end{equation}
We will now split the sum on the zeros in two parts: above and below $T$. The technique is the same for
$h>0$ and $h<0$ but the constants are slightly different, we thus proceed separately for the two cases.
\smallskip\\
Suppose first $h>0$. Under GRH we have
\begin{align*}
\Big|\sum_{|\gamma|\geq T}\frac{(x+h)^{\rho+1}-x^{\rho+1}}{h\rho(\rho+1)}\Big|
%\leq &
%\sum_{|\gamma|\geq T}\Big|\frac{(x+h)^{\rho+1}-x^{\rho+1}}{h\rho(\rho+1)}\Big|\\
\leq &
\sum_{|\gamma|\geq T}
 x^{\frac{3}{2}}\frac{\big(1+\frac{h}{x}\big)^{\frac{3}{2}}+1}{h|\rho(\rho+1)|}
\leq
A\frac{x^{\frac{3}{2}}}{h}\sum_{|\gamma|\geq T}\frac{1}{|\rho^2|},
\end{align*}
with $A:=1+\big(1+\frac{h}{x}\big)^{\frac{3}{2}}$. Thus from~\eqref{eq:E22} and for $T\geq 5$, we get
\[
\Big|\sum_{|\gamma|\geq T}\frac{(x+h)^{\rho+1}-x^{\rho+1}}{h\rho(\rho+1)}\Big|\\
\leq
  \frac{Ax^{\frac{3}{2}}}{\pi Th}
  \Big(\Big(1 + \frac{2.89}{T}\Big)W_\K(T)
    + \Big(1 + \frac{18.61}{T}\Big)n_\K
    +\frac{17.31}{T}
  \Big).
\]
We rewrite
\begin{align*}
\sum_{|\gamma|<T}\frac{(x+h)^{\rho+1}-x^{\rho+1}}{h\rho(\rho+1)}
=&
\sum_{|\gamma|<T}\frac{x^\rho}{\rho}+
\sum_{|\gamma|<T}
            \frac{(x+h)^{\rho+1}-x^{\rho+1}-h(\rho+1)x^\rho}{h\rho(\rho+1)}\\
=&
\sum_{|\gamma|<T}\frac{x^\rho}{\rho}+
hx^{-1/2}\sum_{|\gamma|<T}w_\rho x^{i\gamma}
\end{align*}
with
\[
w_\rho := \frac{\big(1+\frac{h}{x}\big)^{\rho+1}-1-(\rho+1)\frac{h}{x}}
               {\rho(\rho+1)\big(\frac{h}{x}\big)^2}.
\]
From Lemma~\ref{lem:E1} we know that $|w_\rho|\leq\frac{1}{2}$ so that from~\eqref{eq:E21} we deduce
\[
\Big|\sum_{|\gamma|<T}w_\rho x^{i\gamma}\Big|\leq
  \frac{1}{2\pi}\Big[(T+1.45)W_\K(T)
                     - (T - 8.93)n_\K
                     + 8.66
                \Big]
\]
for every $T\geq 5$, giving
\begin{multline*}
\Big|\sum_{\rho}\frac{(x+h)^{\rho+1}-x^{\rho+1}}{h\rho(\rho+1)}
     -\sum_{|\gamma|<T}\frac{x^\rho}{\rho}
\Big|
\leq
  \frac{Ax\sqrt{x}}{\pi Th}\Big[\Big(1+\frac{2.89}{T}\Big)W_\K(T)
                                + \Big(1
                                       + \frac{18.61}{T}
                                  \Big)n_\K
                                + \frac{17.31}{T}
                           \Big] \\
  + \frac{Th}{2\pi\sqrt{x}}\Big[\Big(1+\frac{1.45}{T}\Big)W_\K(T)
                               - \Big(1 - \frac{8.93}{T}\Big)n_\K
                               + \frac{8.66}{T}
                          \Big].
\end{multline*}
The comparison of the main terms suggests taking $h=\frac{2x}{T}$; this brings
$A=1+\big(1+\frac{2}{T}\big)^{3/2}\leq 2+\frac{3}{T}+\frac{3}{2T^2}$
%G 1+(1+4*x/3)^(3/2)+O(x^3)
and
\begin{multline*}
\frac{\pi}{\sqrt{x}}
\Big|\sum_{\rho}\frac{(x+h)^{\rho+1}-x^{\rho+1}}{h\rho(\rho+1)}
     -\sum_{|\gamma|<T}\frac{x^\rho}{\rho}
\Big|
\leq
\Big(1+\frac{3}{2T} +\frac{3}{4T^2}\Big)
                           \Big[\Big(1+\frac{2.89}{T}\Big)W_\K(T)
                                + \Big(1
                                       + \frac{18.61}{T}
                                  \Big)n_\K
                                + \frac{17.31}{T}
                           \Big]                                    \\
   +                       \Big[\Big(1+\frac{1.45}{T}\Big)W_\K(T)
                                - \Big(1 - \frac{8.93}{T}\Big)n_\K
                                + \frac{8.66}{T}
                           \Big].
\end{multline*}
After some simplifications we thus have for $T\geq 5$
\begin{multline}\label{eq:E26}
\frac{\pi}{\sqrt{x}}
  \Big|\sum_{\rho}\frac{(x+h)^{\rho+1}-x^{\rho+1}}{h\rho(\rho+1)}
       -\sum_{|\gamma|<T}\frac{x^\rho}{\rho}
  \Big|
\leq \Big[2
          + \frac{5.84}{T}
          + \frac{5.52}{T^2}
     \Big]W_\K(T)                                                   \\
   + \Big[\frac{29.04}{T}
          + \frac{31.46}{T^2}
     \Big]n_\K
   + \frac{25.97}{T}
   + \frac{28.57}{T^2}.
\end{multline}
%G apply(blr,substpol((1 + 3/2*U + 3/4*U^2)*(1+2.89*U) + U*(1/U+1.45),U^3,U^2/5))        \\elimino la parte in 1/T^3
%G \\%37 = 5.520*U^2 + 5.840*U + 2

%G apply(blr,substpol((1 + 3/2*U + 3/4*U^2)*(1 + 18.61*U) - U*(1/U - 8.76),U^3,U^2/5))   \\elimino la parte in 1/T^3
%G \\%38 = 31.46*U^2 + 29.04 *U

%G apply(blr,substpol((1 + 3/2*U + 3/4*U^2)*5.51*U + 2.76*U,U^3,U^2/5))                  \\elimino la parte in 1/T^3
%G %41 =  9.100*U^2 + 8.270*U

\noindent For $h<0$ the computation is similar with only a few differences. We now have $A\leq 2$ and
$|w_\rho|\leq \frac{1}{2}+\frac{|h|}{6x}$ from Lemma~\ref{lem:E1}. Thus
\begin{multline*}
\Big|\sum_{\rho}\frac{(x+h)^{\rho+1}-x^{\rho+1}}{h\rho(\rho+1)}
     -\sum_{|\gamma|<T}\frac{x^\rho}{\rho}\Big|
\leq
  \frac{2x\sqrt{x}}{\pi T|h|}\Big[\Big(1+\frac{2.89}{T}\Big)W_\K(T)
                                   + \Big(1
                                          + \frac{18.61}{T}
                                     \Big)n_\K
                                   + \frac{17.31}{T}
                             \Big]                                            \\
  + \frac{T|h|}{2\pi\sqrt{x}}\Big(1+\frac{|h|}{3x}\Big)\Big[\Big(1+\frac{1.45}{T}\Big)W_\K(T)
                                                            - \Big(1 - \frac{8.93}{T}\Big)n_\K
                                                            + \frac{8.66}{T}
                                                      \Big].
\end{multline*}
The situation is the same, thus we similarly take $h=-\frac{2x}{T}$ (we then have $x+h>1$ if $T\geq 5$),
producing
\begin{multline*}
\frac{\pi}{\sqrt{x}}
\Big|\sum_{\rho}\frac{(x+h)^{\rho+1}-x^{\rho+1}}{h\rho(\rho+1)}
     -\sum_{|\gamma|<T}\frac{x^\rho}{\rho}\Big|
\leq
  \Big[\Big(1+\frac{2.89}{T}\Big)W_\K(T)
       + \Big(1 + \frac{18.61}{T}\Big)n_\K
       + \frac{17.31}{T}
  \Big]                                                                       \\
 + \Big(1+\frac{2}{3T}\Big)\Big[\Big(1+\frac{1.45}{T}\Big)W_\K(T)
                                - \Big(1 - \frac{8.93}{T}\Big)n_\K
                                + \frac{8.66}{T}
                           \Big]
\end{multline*}
and after some simplifications we get
\begin{multline}\label{eq:E27}
\frac{\pi}{\sqrt{x}}
  \Big|\sum_{\rho}\frac{(x+h)^{\rho+1}-x^{\rho+1}}{h\rho(\rho+1)}
       -\sum_{|\gamma|<T}\frac{x^\rho}{\rho}
  \Big|                                                                       \\
\leq
   \Big[2
        +\frac{5.01}{T}
        +\frac{0.97}{T^2}
   \Big]W_\K(T)
   + \Big[\frac{26.88}{T}
         +\frac{5.96}{T^2}
     \Big]n_\K
   + \frac{25.97}{T}
   + \frac{5.78}{T^2}.
\end{multline}
%G apply(blr,substpol(1+2.89*U + U*(1+2/3*U)*(1/U + 1.45),U^3,U^2/5))
%G \\%43 = 0.9700*U^2 + 5.010*U + 2
%G apply(blr,substpol(1 + 18.61*U -U*(1+2/3*U)* (1/U - 8.93),U^3,U^2/5))
%G \\%44 = 5.960*U^2 + 26.88*U
%G apply(blr,substpol(5.51*PiG*U + U*(1+2/3*U)* 2.76*PiG,U^3,U^2/5))
%G \\%45 = 1.840*PiG*U^2 + 8.270*PiG*U
%
Let $M_{W,\pm}(T)$, $M_{n,\pm}(T)$ and $M_{c,\pm}(T)$ be the functions of $T$ such that the right-hand
side of \eqref{eq:E26} and \eqref{eq:E27} respectively are
\begin{align*}
&M_{W,+}(T)W_\K(T) + M_{n,+}(T) n_\K + M_{c,+}(T)                                     \\
&M_{W,-}(T)W_\K(T) + M_{n,-}(T) n_\K + M_{c,-}(T),
\end{align*}
and let their differences be denoted as
\begin{align*}
D_W(T)    &:=M_{W,+}(T)-M_{W,-}(T)
            =  \frac{0.83}{T} + \frac{4.55}{T^2}                                      \\
D_n(T)    &:=M_{n,+}(T)-M_{n,-}(T)
            =  \frac{2.16}{T}
             + \frac{25.50}{T^2}                                                      \\
D_c(T)    &:=M_{c,+}(T)-M_{c,-}(T)
            = \frac{22.79}{T^2}.
\end{align*}
By~(\ref{eq:E24}--\ref{eq:E27}) we have
\begin{multline}\label{eq:E28}
\Big|\psi_\K(x)-x-\sum_{|\gamma|<T}\frac{x^\rho}{\rho}\Big|
\leq \frac{\sqrt{x}}{\pi}\big(M_{W,+}(T)W_\K(T) + M_{n,+}(T) n_\K + M_{c,+}(T)\big)   \\
     + \frac{x}{T}
     + |r_\K|
     + 1.22\frac{\delta_{n_\K\leq 2}}{x}
     + \max\Big(0,d_\K\log x - \frac{\sqrt{x}}{\pi}\big(D_{W}(T)W_\K(T) + D_{n}(T) n_\K + D_{c}(T)\big)\Big).
\end{multline}
The last term is bounded by $\epsilon_\K(x,T)$, since $D_{c}(T)$ is positive and
$\frac{1}{n_\K}D_{W}(T)W_\K(T) + D_{n}(T)\geq 1.44\pi/T$ when $T\geq 5$.
%G floor((ploth(T=5,100,((0.83/T + 4.55/T^2)*log(T/(2*Pi)) + 2.16/T + 25.50/T^2)*T/Pi)[3])*100)/100.
%G \\%237 = 1.440
%
Moreover, by~\eqref{eq:E23} we have
\begin{equation}\label{eq:E29}
\Big|\sum_{\substack{\rho\\|\gamma|<T}}\frac{x^\rho}{\rho}\Big|
\leq
\frac{\sqrt{x}}{\pi}\Big[
  \Big(\log\Big(\frac{T}{2\pi}\Big) + 4.01\Big)\log \disc_\K
+ \Big(\frac{1}{2}\log^2\Big(\frac{T}{2\pi}\Big) - 1.41\Big)n_\K
+ 25.57
\Big]
\end{equation}
for $T\geq 5$, thus the claim follows from~\eqref{eq:E28}, \eqref{eq:E29} and the upper bound for
$|r_\K|$ in~\eqref{eq:E14}.
\end{proof}

\begin{proof}[Proof of Corollary~\ref{cor:E1}]
When $\K=\Q$ the claim is weaker than~\eqref{eq:E1}, thus from now on we can assume that $n_\K\geq 2$.
Let the claim in the corollary be written as
\[
|\psi_\K(x)-x| \leq F_{c,\disc}(x)\log\disc_\K + G_{c,n}(x)n_\K.
\]
We prove that for every field the bound coming from the theorem is smaller than the one in the
corollary; i.e., that
\begin{equation}\label{eq:E30}
(F_{c,\disc}(x) - F(x,T))\log\disc_\K + (G_{c,n}(x) - G(x,T))n_\K - H(x,T) \geq 0
\end{equation}
for every $x \geq 100$ and some choice for $T=T(x)$. We set $T(x)=\frac{c\sqrt{x}}{\log x}$ with
$c\in[4.8,8]$ (in this way $T\geq 5$ for every $x\geq 36$). An elementary argument proves that the
left-hand side in~\eqref{eq:E30} is $\sqrt{x}$ times a function which increases in $x$ when $x\geq 100$.
\begin{proof}
Dividing the left-hand side of~\eqref{eq:E30} by $\frac{\sqrt{x}}{\pi}$ we get
{\small
\begin{multline*}
\Big[\frac{1}{2}\log x -\Big[\log\Big(\frac{T}{2\pi}\Big)
                             + 6.01
                             + \frac{5.84}{T}
                             + \frac{5.52}{T^2}
                        \Big]
     + \frac{0.98\pi}{\sqrt{x}}
\Big] \log\disc_\K\\
+ \Big[\frac{1}{8}\log^2 x - \Big[\frac{1}{2}\log^2\Big(\frac{T}{2\pi}\Big)
               + \Big(2 + \frac{5.84}{T} + \frac{5.52}{T^2}\Big)\log\Big(\frac{T}{2\pi}\Big)
               - 1.41
               + \frac{29.04}{T}
               + \frac{31.46}{T^2}
          \Big] + \frac{4.10\pi}{\sqrt{x}} - \frac{\pi}{n_\K\sqrt{x}}H(x,T)
  \Big] n_\K,
\end{multline*}
}

\noindent
whose derivative is
{\small
\begin{multline*}
\Big[\frac{1}{2x} - \Big[\log\Big(\frac{T}{2\pi}\Big)
                         + 6.01
                         + \frac{5.84}{T}
                         + \frac{5.52}{T^2}
                    \Big]'
     - \frac{0.49\pi}{x\sqrt{x}}
\Big] \log\disc_\K\\
+
\Big[\frac{\log x}{4x} - \Big[\frac{1}{2}\log^2\Big(\frac{T}{2\pi}\Big)
                           + \Big(2 + \frac{5.84}{T} + \frac{5.52}{T^2}\Big)\log\Big(\frac{T}{2\pi}\Big)
                           - 1.41
                           + \frac{29.04}{T}
                           + \frac{31.46}{T^2}
                      \Big]'
     - \frac{2.05\pi}{x\sqrt{x}}
     - \Big(\frac{\pi}{n_\K}\frac{H(x,T)}{\sqrt{x}}\Big)'
\Big] n_\K.
\end{multline*}
}

\noindent
Since $T'>0$ for $x\geq e^2$, the function $-\frac{\log(T/2\pi)}{T^2}$ is increasing for
$\frac{\sqrt{x}}{\log x} \geq \frac{2\pi\sqrt{e}}{c}$, and since $c\in[4.8,8]$, it is satisfied for every
$x\geq 100$. Moreover,
\[
-\pi\frac{\epsilon_\K(x,T)}{n_\K\sqrt{x}}
 = -\pi\max\Big(0,\frac{d_\K}{n_\K}-\frac{1.44}{c}\Big)\frac{\log x}{\sqrt{x}}
\]
increases for $x\geq e^2$ for every combination of $d_\K, n_\K$ and $c$. Thus, removing some increasing
terms it is sufficient to prove that
%
%{\small
%\begin{multline*}
%\Big(\frac{1}{2x} - \Big(\log\Big(\frac{T}{2\pi}\Big)\Big)' - \Big(\frac{5.84}{T}\Big)' - \frac{0.98\pi/2}{x\sqrt{x}}\Big) \frac{\log\disc_\K}{n_\K}
% + \frac{\log x}{4x}
% - \Big(\frac{1}{2}\log^2\Big(\frac{T}{2\pi}\Big) + 2\log\Big(\frac{T}{2\pi}\Big)\Big)'
% - \Big(\frac{29.04}{T}\Big)'\\
% - 5.84\Big(\frac{1}{T}\log\Big(\frac{T}{2\pi}\Big)\Big)'
% - \frac{4.10\pi/2}{x\sqrt{x}}
% - \Big(\frac{\pi}{n_\K}\frac{\sqrt{x}}{T}\Big)'
%\geq 0
%\end{multline*}
%}
%
%
{\small
\begin{multline*}
\Big[\frac{1}{2x}
     - \Big(\log\Big(\frac{T}{2\pi}\Big)\Big)'
     - \Big(\frac{5.84}{T}\Big)'
     - \frac{0.49\pi}{x\sqrt{x}}
\Big] \log\disc_\K
 +
\Big[\frac{\log x}{4x}
     - \frac{1}{2}\Big(\log^2\Big(\frac{e^2T}{2\pi}\Big)\Big)'
     - \Big(\frac{\pi}{n_\K}\frac{\sqrt{x}}{T}\Big)'\\
     - 5.84\Big(\frac{1}{T}\log\Big(\frac{T}{2\pi}\Big)\Big)'
     - \Big(\frac{29.04}{T}\Big)'
     - \frac{2.05\pi}{x\sqrt{x}}
\Big] n_\K
\geq 0,
\end{multline*}
}
\noindent
which after some computations becomes
%%
%{\small
%\begin{multline*}
%\Big(\frac{1}{2x} - \frac{T'}{T} + \frac{5.84}{T}\frac{T'}{T} - \frac{0.98\pi/2}{x\sqrt{x}}\Big) \frac{\log\disc_\K}{n_\K}
% + \frac{\log x}{4x}
% - \log\Big(\frac{e^2 T}{2\pi}\Big)\frac{T'}{T}
% - \Big(\frac{\pi}{n_\K}\frac{\sqrt{x}}{T}\Big)'\\
% + 5.84\Big(\frac{1}{T}\log\Big(\frac{T}{2\pi e}\Big)\Big)\frac{T'}{T}
% - \frac{29.04}{c}\Big(\frac{\log x}{\sqrt{x}}\Big)'
% - \frac{4.10\pi/2}{x\sqrt{x}}
%\geq 0
%\end{multline*}
%}

%%
%{\small
%\begin{multline*}
%\Big(\frac{1}{2x} - \Big(1 - \frac{5.84}{c}\frac{\log x}{\sqrt{x}}\Big)\frac{1}{2x}\Big(1-\frac{2}{\log x}\Big) - \frac{0.98\pi/2}{x\sqrt{x}}\Big) \frac{\log\disc_\K}{n_\K}
% + \frac{\log x}{4x}
% - \log\Big(\frac{e^2 T}{2\pi}\Big)\frac{1}{2x}\Big(1-\frac{2}{\log x}\Big)\\
% - \frac{\pi/c}{n_\K x}
% + 5.84\Big(\frac{\log x}{c\sqrt{x}}\log\Big(\frac{c\sqrt{x}}{2\pi e\log x}\Big)\Big)\frac{1}{2x}\Big(1-\frac{2}{\log x}\Big)
% + \frac{29.04}{c}\frac{\log x-2}{2x\sqrt{x}}
% - \frac{4.10\pi/2}{x\sqrt{x}}
%\geq 0
%\end{multline*}
%}

%%
%{\small
%\begin{multline*}
%\Big(\frac{2}{\log x} + \frac{5.84}{c}\frac{\log x-2}{\sqrt{x}} - \frac{0.98\pi}{\sqrt{x}}\Big) \frac{\log\disc_\K}{n_\K}
% + \frac{\log x}{2}
% - \log\Big(\frac{c e^2}{2\pi}\frac{\sqrt{x}}{\log x}\Big)\Big(1-\frac{2}{\log x}\Big)\\
% - \frac{2\pi}{c n_\K}
% + \frac{5.84}{c\sqrt{x}}\log\Big(\frac{c\sqrt{x}}{2\pi e\log x}\Big)(\log x - 2)
% + \frac{29.04}{c}\frac{\log x-2}{\sqrt{x}}
% - \frac{4.10\pi}{\sqrt{x}}
%\geq 0
%\end{multline*}
%}

%
{\small
\begin{multline*}
\Big[\frac{2}{\log x} + \frac{5.84}{c}\frac{\log x-2}{\sqrt{x}} - \frac{0.98\pi}{\sqrt{x}}
\Big] \log\disc_\K\\
 + \Big[1
        - \frac{2\pi}{c n_\K}
        - \log\Big(\frac{c e^2}{2\pi\log x}\Big)\Big(1-\frac{2}{\log x}\Big)
        + \frac{5.84}{c\sqrt{x}}\log\Big(\frac{c\sqrt{x}}{2\pi e\log x}\Big)(\log x - 2)
        + \frac{29.04}{c}\frac{\log x-2}{\sqrt{x}}
        - \frac{4.10\pi}{\sqrt{x}}
   \Big] n_\K
\geq 0.
\end{multline*}
}
\noindent Recalling the restriction $c\in[4.8,8]$, one proves that both the coefficient of $\log\disc_\K$
and of $n_\K$ are positive for all $x\geq 12$.
%G coeff. of \log\disc_\K
%G my(c=4.8); solve(x=3,10^9,2/log(x)+5.84/c*log(x)/sqrt(x)*(1-2/log(x))-0.98*Pi/sqrt(x))
%G my(c=8);   solve(x=3,10^9,2/log(x)+5.84/c*log(x)/sqrt(x)*(1-2/log(x))-0.98*Pi/sqrt(x))
%G
%G coeff. of n_\K
%G f(x,n,c) = 1 - 2*Pi/c/n - log(c*exp(2)/2/Pi/log(x))*(1-2/log(x)) + 5.84/c/sqrt(x)*log(c/(2*Pi*exp(1))*sqrt(x)/log(x))*(log(x)-2) + 29.04/c*(log(x)-2)/sqrt(x) - 4.10\pi/sqrt(x)
%G \\ the worst case is n_K=1
%G my(n=1,c=8);ploth(x=3,10^4,f(x,n,c))[3]
%G %48 = 0.20265953876403197676
%G my(n=1,c=7);solve(x=3,10^4,f(x,n,c))
%G %49 = 6.305240499196643496849918228
%G my(n=1,c=6);solve(x=3,10^4,f(x,n,c))
%G %50 = 7.798683718441830114997181016
%G my(n=1,c=5);solve(x=3,10^4,f(x,n,c))
%G %51 = 9.293015657553155984225970588
%G my(n=1,c=4);solve(x=3,10^4,f(x,n,c))
%G %52 = 11.08331004452696732525247592
%G
%G
\end{proof}

We further notice that the coefficients of $\log\disc_\K$ and of $n_\K$ in~\eqref{eq:E30} are positive
when $x\geq 100$. In fact they can be written as $\sqrt{x}$ times a monotonous function of $x$ (repeating
the previous argument, this time without the contribution of $H(x,T)$), and their value in $x=100$ is
positive for every $c\in[4.8,8]$. Now we split the argument according to the value of $n_\K$.
\medskip\\
$n_\K\geq 8$.          %
   We are assuming GRH, so $\log\disc_\K \geq n_\K\log(11.916) - 5.8507$ (see~\cite{Odlyzko1,
   OdlyzkoTables,Odlyzko2, Odlyzko3,Odlyzko4} and entry $b=1.6$ of Table~3 in~\cite{OdlyzkoTables}). Thus
   we can prove the claim by proving that
   \[
    (F_{c,\disc}(x) - F(x,T))(n_\K\log(11.916) - 5.8507) + (G_{c,n}(x) - G(x,T))n_\K - H(x,T) \geq 0,
   \]
   and since the coefficient of $n_\K$ is positive, it is sufficient to prove it for $n_\K=8$.
%   \begin{equation}%\label{eq:A18}
%    (F_{c,\disc}(x) - F(x,T))(8\log(11.916) - 5.8507) + (G_{c,n}(x) - G(x,T))8 - H(x,T) \geq 0.
%   \end{equation}
   We set $c=8$. We have verified that the left-hand side is $\sqrt{x}$ times an increasing function (for
   $x\geq 100$), thus the inequality can be proved for every $x\geq 100$ simply by testing its value in
   $x=100$.
\medskip\\
$n_\K=5$, $6$ and $7$. %
   We repeat the previous argument, but now with the minimal discriminants which are $1609$, $9747$ and
   $184607$, respectively (see~\cite[Table~1]{Odlyzko4}).
\medskip\\
$2\leq n_\K\leq 4$.    %
   For every such degree one checks that~\eqref{eq:E30} holds when $\disc_\K > \overline{\disc_\K}$ where
   $\overline{\disc_\K}$ is in Table~\ref{tab:E1} (by monotonicity in $x$ it is sufficient to check the
   claim for $x=100$); we adjust the parameter $c$ to get a smaller $\overline{\disc_\K}$.
   \begin{table}[H]
   \caption{Minimal discriminants $\overline{\disc_\K}$ for~\eqref{eq:E30}.}\label{tab:E1}
   \smallskip
    \begin{tabular}{l|r|r|r|r|r|r}
    %  \toprule
      $r_2\backslash n_\K$ & $2$ ($c=4.8$) & $3$ ($c=5.1$)           & $4$ ($c=6$) \\
      \midrule
                       $0$ & $172921407$   & \phantom{$00$}$1350275$ & $10311$     \\
                       $1$ & $103995324$   & $ 369421$               & $ 2584$     \\
                       $2$ &               &                         & $  648$     \\
    %  \bottomrule
    \end{tabular}
   \end{table}
   \noindent %
   This proves the claim for all fields but those with $n_\K \leq 4$ and $\disc_\K\leq
   \overline{\disc_\K}$. Actually, all fields with small degree and small discriminants are
   known~\cite{MegrezTables}
%   for cubic fields we have used the algorithm of Karim Belabas~\cite{Belabas};
   (for quadratic fields we use the fundamental discriminants below $\overline{\disc_\K}$), and the number
   of these exceptions is in Table~\ref{tab:E2}.
   \begin{table}[H]
   \caption{Number of exceptional fields for~\eqref{eq:E30}.}\label{tab:E2}
   \smallskip
    \begin{tabular}{l|r|r|r|r|r|r}
    %  \toprule
      $r_2\backslash n_\K$ & $2$        & $3$                    & $4$                    \\
      \midrule
                       $0$ & $52561764$ & \phantom{$000$}$74747$ & \phantom{$000000$}$54$ \\
                       $1$ & $31610787$ & $65708$                & $73$                   \\
                       $2$ &            &                        & $22$                   \\
    %  \bottomrule
    \end{tabular}
   \end{table}
   \noindent %
   For each exceptional field we come back to~\eqref{eq:E30} and prove it for every $x\geq \bar{x}$ in
   Table~\ref{tab:E3} (using again the monotonicity in $x$); we adjust the parameter $c$ to get a
   smaller $\bar{x}$.
   \begin{table}[H]
   \caption{Minimal $x$ for the exceptional fields for~\eqref{eq:E30}; the minimal discriminants come
   from~\cite[Table~1]{Odlyzko4}; $\bar{x}$ is the one associated with the smallest discriminant.}
   \label{tab:E3}
   \smallskip
    \begin{tabular}{l|r|r|r|r|r|r}
    %  \toprule
      $n_\K$             & $2$ ($c= 4.8$)& $3$ ($c=5$)           & $4$ ($c=5$)           \\
      \midrule
      minimal $\disc_\K$ & $3$           & \phantom{$0000$}$23 $ & \phantom{$0000$}$117$ \\
      $\bar{x}$          & $1566020$     & $980$                 & $184$                 \\
    %  \bottomrule
    \end{tabular}
   \end{table}
   \noindent%
   At last we test the claim for the exceptional fields in the exceptional range in Table~\ref{tab:E3} by
   computing $|\psi_\K(x)-x|$ (with PARI/GP~\cite{PARI2}) and by checking that the difference with the
   bound is at least $1$: in this way we only need to check the integers $x$ in the range. This idea
   works for the fields in our list of degree $3$ and $4$. For quadratic fields both the number of fields
   and $\bar{x}$ are much larger. Luckily, the value of $\bar{x}$ drops down quickly when the
   discriminant increases, and for discriminants larger that $100$ it is already only $5040$, which can
   be checked very fast. Therefore the really long computations are only those for quadratic fields with
   discriminants below $100$. The entire check can be made in approximately 40 hours on a 2011 personal
   computer.
\end{proof}

\begin{proof}[Proof of Corollary~\ref{cor:E2}]
In~\eqref{eq:E4} we make the choice $T = \frac{10}{e} \frac{\sqrt{x}}{\log x}$, for which the condition
$T\geq 5$ is satisfied for every $x\geq 3$. The term $\epsilon_\K(x,T)$ in Theorem~\ref{th:E1} is $\leq
0.61 d_\K\log x$, and
\begin{align*}
F(x,T) \leq& \frac{\sqrt{x}}{\pi}\Big[\frac{1}{2}\log\Big(x \frac{25e^2}{\pi^2}\frac{e^{\frac{11.68}{T} + \frac{11.04}{T^2}}}{\log^2 x}\Big)
                                      + 4.01
                                 \Big]
              + 1.02,                                                                                    \\
%G 2*(5.84*U+5.52*U^2)
%G \\%97 = 11.040*U^2 + 11.680*U
%
%
G(x,T) \leq& \frac{\sqrt{x}}{\pi}\Big[\frac{1}{8}\log^2\Big(x\frac{25e^2}{\pi^2}\frac{e^{\frac{11.68}{T}+\frac{11.04}{T^2}}}{\log^2 x}\Big)
                                      - 3.41
                                      + \frac{17.36}{T}
                                      + \frac{3.37}{T^2}
                                      - \frac{32.23}{T^3}
                                      - \frac{15.23}{T^4}
                                  \Big]
             -2.10,                                                                                       \\
%G -1/2*(2+5.84*U+5.52*U^2)^2-1.41+29.04*U+31.46*U^2
%G %13 = -15.23520*U^4 - 32.236800*U^3 + 3.36720*U^2 + 17.360*U - 3.410
%
%
H(x,T) \leq& \frac{e}{10} \sqrt{x}\log x
          + 25.57\frac{\sqrt{x}}{\pi}
      + 0.61d_\K\log x
          + 2.75\log x
          + 8.76.
%G my(T=5);(8.27+9.10/T)*exp(1)/10
%G %14 = 2.742746364915176642478530059
%G 8.35 + 1.22/3
%G \\%40 = 8.7566666666666666666666666666666666667
\end{align*}
The first claim in Corollary~\ref{cor:E2} follows plugging these bounds in~\eqref{eq:E3}, after some simplifications.
%G F(x,T)=sqrt(x)/Pi*(1/2*log(x*25*exp(2+11.68/T+11.04/T^2)/Pi^2/log(x)^2)+4.01) +1.02
%G ploth(x=3,1000,(F(x,10/exp(1)*sqrt(x)/log(x))           -sqrt(x)*1/2/Pi*log(18.8*x/log(x)^2))/sqrt(x))[4]
%G \\ %25 = 2.237356019008694741
%G G(x,T)=sqrt(x)/Pi*(1/8*log(x*25*exp(2+11.68/T+11.04/T^2)/Pi^2/log(x)^2)^2-3.41+17.36/T+3.37/T^2-32.23/T^3-15.23/T^4) - 2.10
%G ploth(x=3,1000,(G(x,10/exp(1)*sqrt(x)/log(x))+0.61*log(x)-sqrt(x)*1/8/Pi*log(18.8*x/log(x)^2))/sqrt(x))[4]
%G \\ %29 =1.2277717900236981841
%G H(x,T)=x/T+25.57*sqrt(x)/Pi-0.61*log(x)+2.75*log(x)+8.76
%G ploth(x=3,1000,(H(x,10/exp(1)*sqrt(x)/log(x))-0.3*sqrt(x)*log(x))/sqrt(x))[4]
%G %33 = 14.523190218579836497
%
\noindent
For the second inequality we set $T = \frac{2\pi}{e^2}\frac{\sqrt{x}}{\log x}$; in this case the term
$\epsilon_\K(x,T)$ in Theorem~\ref{th:E1} is $0$,
%G 1-1.44*exp(2)/(2*Pi)
%G \\%41 = -0.6934469162165068962955705797
the condition $T\geq 5$ requires $x\geq 2000$, and the claim follows as the previous one.
\end{proof}

\begin{proof}[Proof of Corollary~\ref{cor:E3}]
Let
\[
\vartheta_\K(x) := \sum_{\substack{\P\\ \Norm\P \leq x}} \log\Norm\P.
\]
Then one has
\[
\pi_\K(x)-\pi_\K(\bar{x})- \int_{\bar{x}}^x \frac{\d u}{\log u}
 =    \int_{\bar{x}}^x \frac{\d(\vartheta_\K(u)-u)}{\log u},
\]
which by partial integration gives
\begin{equation}\label{eq:E31}
\Big|\pi_\K(x)-\pi_\K(\bar{x})- \int_{\bar{x}}^x \frac{\d u}{\log u}\Big|
 \leq \int_{\bar{x}}^x \frac{\d|\vartheta_\K(u)-u|}{\log u}
 \leq \frac{|\vartheta_\K(x)-x|}{\log x} + \int_{\bar{x}}^x \frac{|\vartheta_\K(u)-u|\d u}{u\log^2 u}.
\end{equation}
Moreover, there are at most $n_\K$ ideals of the form $\P^m$ ($\P$ prime) of a given norm in $\K$, so
\[
|\psi_\K(x)-\vartheta_\K(x)|
\leq n_\K |\psi_\Q(x)-\vartheta_\Q(x)|
\leq 1.43\,n_\K \sqrt{x},
\]
where the last inequality is Theorem~13 in~\cite{RosserSchoenfeld}.
%
%\[
%\sum_{\substack{\P,m\geq 2\\ \Norm\P^m \leq x}}\log\Norm\P \leq n_\K
%\sum_{\substack{p,m\geq 2\\ p^m \leq x}}\log p =    n_\K \sum_{k=2}^{\intpart{\log_2 x}} \vartheta_\Q(x^{1/k}).
%\]
%Theorem~13 in~\cite{RosserSchoenfeld} gives $\sum_{k=2}^{\intpart{\log_2 x}} \vartheta_\Q(x^{1/k}) =
%|\psi_\Q(x)-\vartheta_\Q(x)|\leq 1.43 \sqrt{x}$ for every $x$, thus
%\[
%|\psi_\K(x)-\vartheta_\K(x)|
%\leq 1.43 \sqrt{x}\,n_\K
%\qquad
%\forall\,x.
%\]
This shows that $\vartheta_\K(x)$ satisfies the same bound of $\psi_\K(x)$, at the cost of adding $1.43
n_\K \sqrt{x}$. Substituting this bound and the first inequality in Corollary~\ref{cor:E2}
into~\eqref{eq:E31} and after some numerical approximations one gets the corollary.
\end{proof}

\begin{table}[H]
\caption{Constants for $\sum_{|\gamma|\leq 5} |\rho|^{-1}$ in Lemma~\ref{lem:E3}.}\label{tab:E4}
\smallskip
{\scriptsize
\begin{tabular}{|r|r||r|r|}
  \toprule
  $j$     & \mltc{1}{c||}{$a_j\cdot 10^7$} &  $j$      &  \mltc{1}{c|}{$a_j\cdot 10^7$}   \\
  \midrule
  $ 1$    & $                    -324328089$&  $25$    & $ -52154912212245427675107284117$\\
  $ 2$    & $                  115693093357$&  $26$    & $  72227309752304735434420743120$\\
  $ 3$    & $               -10579381239203$&  $27$    & $ -91546659026910381192366828396$\\
  $ 4$    & $               495540769876127$&  $28$    & $ 106117853961289012764032450733$\\
  $ 5$    & $            -14528281352885983$&  $29$    & $-112369546004525999862866475251$\\
  $ 6$    & $            296347058332550155$&  $30$    & $ 108533470948598920563558219043$\\
  $ 7$    & $          -4498154499661073603$&  $31$    & $ -95431698456287244651252772381$\\
  $ 8$    & $          53248447239339829090$&  $32$    & $  76206788473674179730998288621$\\
  $ 9$    & $        -508947342104081739447$&  $33$    & $ -55105812322315804526845019881$\\
  $10$    & $        4033084416071505510477$&  $34$    & $  35955970546002972861665837368$\\
  $11$    & $      -27051470635668143949707$&  $35$    & $ -21079935102298710141936369413$\\
  $12$    & $      156121546937577920978167$&  $36$    & $  11047616237574616067334355219$\\
  $13$    & $     -785529078417852387859619$&  $37$    & $  -5143709248575449263188160534$\\
  $14$    & $     3482495472267374521416188$&  $38$    & $   2111566552644017238627810350$\\
  $15$    & $   -13720533216155265613103988$&  $39$    & $   -757162365842762640320305866$\\
  $16$    & $    48375037637788872322025183$&  $40$    & $    234379624034767935847527151$\\
  $17$    & $  -153492067547835461489301521$&  $41$    & $    -61692234538384117080736694$\\
  $18$    & $   440289327629182231371781424$&  $42$    & $     13534020670767148307863583$\\
  $19$    & $ -1145934878685670756527108765$&  $43$    & $     -2407266538638620726296042$\\
  $20$    & $  2713965041058219158192688004$&  $44$    & $       333452115133845423979326$\\
  $21$    & $ -5861973594145453618923885659$&  $45$    & $       -33740880236473501034280$\\
  $22$    & $ 11566694720865120123031709900$&  $46$    & $         2218003445878553284287$\\
  $23$    & $-20874589384842483010331503670$&  $47$    & $          -71076474624305025203$\\
  $24$    & $ 34482298986730410055952580804$&  ---     & \mltc{1}{c|}{---}                \\
  \bottomrule
\end{tabular}
}
\end{table}

%\bibliographystyle{amsplain}
%\bibliography{f:/books}

\begin{thebibliography}{10}
\bibitem{Goldston3}
D.~A. Goldston, \emph{On a result of {L}ittlewood concerning prime numbers},
  Acta Arith. \textbf{40} (1981/82), no.~3, 263--271.

\bibitem{GrenieMolteni2}
L.~Greni\'{e} and G.~Molteni, \emph{Explicit smoothed prime ideals theorems under
{GRH}}, to appear in Math. Comp., \url{http://arxiv.org/abs/1312.4465}, 2015.

\bibitem{Ingham2}
A.~E. Ingham, \emph{The distribution of prime numbers}, Cambridge University
  Press, Cambridge, 1990.

\bibitem{KadiriNg}
H.~Kadiri and N.~Ng, \emph{Explicit zero density theorems for {D}edekind zeta
  functions}, J. Number Theory \textbf{132} (2012), no.~4, 748--775.

\bibitem{LagariasOdlyzko}
J.~C. Lagarias and A.~M. Odlyzko, \emph{Effective versions of the {C}hebotarev
  density theorem}, Algebraic number fields: {$L$}-functions and {G}alois
  properties ({P}roc. {S}ympos., {U}niv. {D}urham, {D}urham, 1975), Academic
  Press, London, 1977, pp.~409--464.

\bibitem{Lang1}
S.~Lang, \emph{Algebraic number theory}, second ed., Springer-Verlag, New York,
  1994.

\bibitem{Littlewood}
J.~E. Littlewood, \emph{Two notes on the {R}iemann {Z}eta-function}, Cambr.
  Phil. Soc. Proc. \textbf{22} (1924), 234--242.

\bibitem{Odlyzko1}
A.~M. Odlyzko, \emph{Some analytic estimates of class numbers and
  discriminants}, Invent. Math. \textbf{29} (1975), 275--286.

\bibitem{OdlyzkoTables}
\bysame, \emph{Discriminant bounds},
  \url{http://www.dtc.umn.edu/~odlyzko/unpublished/index.html}, 1976.

\bibitem{Odlyzko2}
\bysame, \emph{Lower bounds for discriminants of number fields}, Acta Arith.
  \textbf{29} (1976), 275--297.

\bibitem{Odlyzko3}
\bysame, \emph{Lower bounds for discriminants of number fields. {II}}, Tohoku
  Math. J., II. Ser. \textbf{29} (1977), 209--216.

\bibitem{Odlyzko5}
\bysame, \emph{On conductors and discriminants}, Algebraic number fields:
  {$L$}-functions and {G}alois properties ({P}roc. {S}ympos., {U}niv. {D}urham,
  {D}urham, 1975), Academic Press, London, 1977, pp.~377--407.

\bibitem{Odlyzko4}
\bysame, \emph{Bounds for discriminants and related estimates for class
  numbers, regulators and zeros of zeta functions: A survey of recent results},
  Sem. Theorie des Nombres, Bordeaux \textbf{2} (1990), 119--141.

\bibitem{Oesterle}
J.~Oesterl\'{e}, \emph{Versions effectives du th\'{e}or\`{e}me de {C}hebotarev
  sous l'hypoth\`{e}se de {R}iemann g\'{e}n\'{e}ralis\'{e}e}, Ast\'{e}risque
  \textbf{61} (1979), 165--167.

\bibitem{MegrezTables}
The PARI~Group, Bordeaux, \emph{megrez number field tables}, 2008, Package
  \emph{nftables.tgz} from \url{http://pari.math.u-bordeaux.fr/packages.html}.

\bibitem{PARI2}
The PARI~Group, Bordeaux, \emph{{PARI/GP}, version {\tt 2.6.0}}, 2013,
  available from \url{http://pari.math.u-bordeaux.fr/}.

\bibitem{RosserSchoenfeld}
J.~B. Rosser and L.~Schoenfeld, \emph{Approximate formulas for some functions
  of prime numbers}, Illinois J. Math. \textbf{6} (1962), 64--94.

\bibitem{RosserSchoenfeld2}
\bysame, \emph{Sharper bounds for the {C}hebyshev functions {$\theta (x)$} and
  {$\psi (x)$}}, Math. Comp. \textbf{29} (1975), 243--269.

\bibitem{Schoenfeld1}
L.~Schoenfeld, \emph{Sharper bounds for the {C}hebyshev functions {$\theta
  (x)$} and {$\psi (x)$}. {II}}, Math. Comp. \textbf{30} (1976), no.~134,
  337--360, \emph{Corrigendum} in Math. Comp. \textbf{30} (1976), no.~136, 900.

\bibitem{Stark1}
H.~M. Stark, \emph{Some effective cases of the {B}rauer-{S}iegel theorem},
  Invent. Math. \textbf{23} (1974), 135--152.

\bibitem{TrudgianIII}
T.~Trudgian, \emph{An improved upper bound for the error in the zero-counting
  formulae for {D}irichlet {L}-functions and {D}edekind zeta-functions},
  Math. Comp. \textbf{84} (2015), no.~293, 1439--1450.

\bibitem{Winckler}
B.~Winckler, \emph{Th\'{e}or\`{e}me de {C}hebotarev effectif}, arxiv:1311.5715,
  \url{http://arxiv.org/pdf/1311.5715v1.pdf}, 2013.
\end{thebibliography}
%\end{document}

\end{document}